\documentclass{article}
\usepackage[T1]{fontenc}
\usepackage[cp1250]{inputenc}
\usepackage{amsthm}
\usepackage{amsmath}
\usepackage{amsfonts}
\usepackage{multicol}
\usepackage{epsfig}
\usepackage{verbatim}
\usepackage{hyperref}
\usepackage[affil-it]{authblk}
\usepackage{multirow}
\usepackage{natbib}

\newtheorem{theorem}{Theorem}
\newtheorem{lemma}{Lemma}

\theoremstyle{definition}
\newtheorem{example}{Example}
\theoremstyle{plain}

\newtheorem{proposition}{Proposition}
\theoremstyle{plain}

\def\Var{\mathrm{Var}}
\def\diag{\mathrm{diag}}
\def\T{T}

\def\diag{\mathrm{diag}}
\def\Var{\mathrm{Var}}

\title{Optimal Approximate Designs for Comparison with Control in Dose-Escalation Studies}
\author{Samuel Rosa and Radoslav Harman}
\date{\today}

\begin{document}
	
	\maketitle
	
	\section*{Abstract}
	
	Consider an experiment, where a new drug is tested for the first time on human subjects - healthy volunteers. Such experiments are often performed as dose-escalation studies: a set of increasing doses is pre-selected, individuals are grouped into cohorts, and in each cohort, the dose number $i$ can be administered only if the dose number $i-1$ has already been tested in the previous cohort. If an adverse effect of a dose is observed, the experiment stops and thus no subjects are exposed to higher doses. In this paper, we assume that the response is affected both by the dose or placebo effects as well as by the cohort effects. We provide optimal approximate designs for selected optimality criteria ($E$-, $MV$- and $LV$-optimality) for estimating the effects of the drug doses compared with the placebo. In particular, we obtain the optimality of Senn designs and extended Senn designs with respect to multiple criteria.
	\bigskip
	
	\emph{Keywords:} dose-escalation studies; approximate designs; comparison with control; optimal designs; block designs

\section{Introduction}

A dose-escalation study for a new drug consists of a series of trials, in which the drug is given to cohorts of individuals, one cohort at a time. The experimenters select a finite number of increasing doses of the drug that are to be tested and each subject is given either one of the doses or the placebo. 
In particular, the subjects in the first cohort may be given only the lowest dose or the placebo, the subjects in the second cohort only the two lowest doses or the placebo, etc. If an adverse effect is observed, the experiment stops; therefore, no subjects are exposed to higher doses. These studies are often performed when a new drug is tried on humans for the first time and thus the safety of the subjects is the main concern, especially if the drug is given to healthy volunteers - hence the slowly escalating doses.

\cite{Bailey09}, motivated by the failure of the Te Genero trial (see \cite{SennRSS}), considered designs for dose-escalation studies with quantitative responses. \cite{Bailey09} modelled the dose-escalation studies as blocking experiments with design constraints, where the blocks represented the cohorts, and the treatments were the placebo and the drug doses.  Some general principles that reduce variance were presented, as well as a wide range of well performing designs for estimating a system of all pairwise comparisons of treatments.
\cite{HainesClark} provided some numerical methods for constructing $A$-, $MV$-, $D$- and $E$-optimal designs in the setting proposed by \cite{Bailey09}.

In this paper, we adopt the model considered by \cite{Bailey09} and we study the optimal approximate designs. The approximate, or continuous, designs do not determine the actual numbers of subjects that are given the particular treatments; rather, they determine the proportions of all subjects that are allocated to each combination of a cohort and a treatment. Note that when a total number of subjects is selected, optimal or efficient exact designs can usually be constructed from the optimal approximate designs, e.g., see Chapter 12 in \cite{puk}.
\bigskip 

Senn in his commentary \cite{Senn09} argued that in each cohort the comparison of the latest dose with the placebo is relevant, and proposed a corresponding latest variance ($LV$) optimality criterion. This relevance, as noted by \cite{Senn09}, follows from the fact that the crucial decision whether to proceed to the next cohort or to stop the experiment is based on the effect of the latest dose relative to the placebo. Following the argued relevance of the comparisons with the placebo, we consider the problem of comparing doses of the drug with the placebo, unlike \cite{Bailey09}, who studied the set of all pairwise comparisons. Therefore, the system of contrasts studied here is, in fact, a comparison of test treatments (drug doses) with a control (placebo).

There is a large amount of literature on comparison with control in blocking experiments, beginning with \cite{BechhoferTamhane} (e.g., see \cite{MajumdarNotz}, \cite{Majumdar} and \cite{KunertMartin}). However, these standard results are not applicable in the model considered in this paper, due to the dose-escalation constraints.

We consider the $E$- and $MV$-optimality criteria, as well as the latest variance ($LV$) criterion proposed by \cite{Senn09}, and we obtain optimal designs for these criteria. In particular, the \emph{Senn designs} (which are designs that assign in each cohort half of the subjects to the placebo and the other half to the latest dose, see \cite{Bailey09}, \cite{Senn97}) stand out, because they are optimal with respect to all of the aforementioned criteria. Moreover, we obtain a complete characterization of the class of all $E$-optimal designs, which allows for selecting the best $E$-optimal design with respect to a secondary criterion chosen by the experimenter.


\subsection{Notation}

The symbols $1_n$ and $0_n$ denote the column vectors of length $n$ of ones and zeroes, respectively. Furthermore, $J_n := 1_n1_n^\T $ is the $n \times n$ matrix of ones, $0_{m \times n} := 0_m 0_n^\T $ is the $m \times n$ matrix of zeroes and $e_i$ is the $i$-th standard unit vector (the $i$-th column of the identity matrix $I_n$, where $n$ is the dimension of $e_i$). We use the symbols $\lambda_{\min}(A)$ and $\lambda_{\max}(A)$ to denote the smallest and the largest eigenvalue, respectively, of a symmetric matrix $A$.
We denote the column space of a matrix $A$ as $\mathcal{C}(A)$.

\section{The Model}

We consider the model formulated by \cite{Bailey09}. A new drug is to be tested on $N$ individuals split into $t$ cohorts. Each individual is given a treatment $i$, which is either the placebo (for $i=0$) or one of $n \geq 2$ increasing doses of the drug (for $i=1,\ldots,n$), i.e., $\mathrm{dose}_1 < \ldots < \mathrm{dose}_n$. We assume that the safety or tolerability outcome, or in general, the measured effects of the drug on the subjects, are quantitative. Furthermore, we assume additivity of treatment and cohort effects, i.e.,
\begin{equation}\label{eModelDE}
Y_j = \mu + \tau_{i(j)} + \theta_{k(j)} + \varepsilon_j, \quad j=1,\ldots,N,
\end{equation}
where $N$ is the total number of subjects, $Y_j$ is the measured safety or tolerability response for the $j$-th subject, $\mu$ is the overall mean, $i(j) \in \{0,\ldots,n\}$ is the treatment given to the $j$-th subject ($0$ denotes the placebo, $1, \ldots, n$ denote the escalating doses), $k(j) \in \{1,\ldots,t\}$ is the cohort in which the $j$-th subject is; $\tau_i$ is the effect of the $i$-th treatment, $\theta_k$ is the effect of the $k$-th cohort, and $\varepsilon_j$ is the random error. The errors $\{\varepsilon_j\}$ are assumed to have zero mean, variance $\sigma^2$ and they are assumed to be independent. Note that the model (\ref{eModelDE}) is in fact a blocking experiment, with cohorts as blocks, and the doses and the placebo as treatments. Adopting the notation of \cite{Bailey09}, we will consider two cases: standard designs, where $t=n$ and extended designs, where $t=n+1$.

Based on the relevance of comparing the latest dose with the placebo argued in \cite{Senn09}, we assume that the objective of the experiment is to compare the drug doses with the placebo; i.e., to estimate $\tau_1-\tau_0, \ldots, \tau_n-\tau_0$. 
\bigskip

Consider the general linear model
$$Y_j = f^\T (x_j) \beta + \varepsilon_j, \quad j=1,\ldots,N,$$
where $x_j \in \mathfrak{X}$, which is the finite set of all experimental conditions. Then, an \emph{approximate design} (from now on, simply a \emph{design}) $\xi$ is a mapping from $\mathfrak{X}$ to $[0,1]$, such that $\sum_{x \in \mathfrak{X}} \xi(x) =1$.  The moment matrix of a design $\xi$ is $M(\xi) = \sum_{x \in \mathfrak{X}} \xi(x)f(x)f^\T (x)$. A design $\xi$ is said to be feasible for a system $A^T\beta$ if $\mathcal{C}(A) \subseteq \mathcal{C}(M(\xi))$; in such a case, the system $A^T\beta$ is said to be estimable under $\xi$. The information matrix of a feasible design $\xi$ for estimating $A^\T \beta$ is $N_A(\xi) = (A^\T  M^-(\xi) A)^{-1}$, see \cite{puk}.

A feasible design $\xi^* \in \Xi$, where $\Xi$ is the set of all considered designs, is $\Phi$-optimal if it maximizes (alternatively, minimizes) a given function $\Phi(\xi)$ among all designs $\xi \in \Xi$. We will omit the argument $\xi$ in expressions, in which it is clear from the context.

In model (\ref{eModelDE}), we have $\mathfrak{X} = \{0,\ldots,n\} \times \{1,\ldots,t\}$, $f(i,k) = \big(e_{i+1}^\T , 1 , e_k^\T  \big)^\T $, $\beta = \big(\tau^\T , \mu, \theta^\T \big)^\T $, where $\tau =(\tau_0, \ldots, \tau_n)^\T $ and $\theta=(\theta_1, \ldots, \theta_t)^\T $. The objective of the experiment is to estimate  $\tau_1-\tau_0, \ldots, \tau_n-\tau_0$, which can be written as a system of contrasts $Q^\T  \tau$, where  $Q^\T =(-1_{n},I_{n})$, or, equivalently, $A^\T  \beta$, where $A^\T  = \big(Q^\T , 0_{n \times (t+1)} \big)$.

The value $\xi(i,k)$ determines the proportion of all subjects that are assigned to cohort $k$ and given the treatment $i$.
Let $r_i(\xi) := \sum_k \xi(i,k)$ be the total proportion of the trials with treatment $i$, $i=0,\ldots,n$, $s_k(\xi) :=\sum_i \xi(i,k)$ be the total proportion of the trials for cohort $k$, $k=1,\ldots,t$ and $r(\xi):=(r_0(\xi),\ldots,r_n(\xi))^\T $, $s(\xi):=(s_1(\xi),\ldots,s_t(\xi))^\T $. Then, $Nr_i$ is the replication number of the $i$-th treatment, and $Ns_k$ is the size of the $k$-th cohort. Thus, $r$ is the vector of relative replication numbers and $s$ is the vector of relative sizes of the cohorts.
The moment matrix of a design $\xi$ is
$$
M(\xi)=\begin{bmatrix}
M_{11} & M_{12} \\ M_{12}^\T  & M_{22}
\end{bmatrix},
$$
where $M_{11}=\diag(r)$; $M_{12}=\big(r, X\big)$, where $X=\big(\xi(i,k)\big)_{i,k}$; and
$$
M_{22} = \begin{bmatrix}
1 & s^\T  \\ s & \diag(s)
\end{bmatrix}.
$$
Note that $r=X1_t$ and $s=X^\T 1_{n+1}$.
\bigskip

In accordance with \cite{Bailey09}, we introduce some constraints on the set of considered designs $\Xi$. For each $k\in \{1,\ldots,n\}$, we have:
\begin{itemize}
	\item[(i)] $\xi(i,k)=0$ for $i>k$. In cohort $k$, the highest allowed dose is $k$.
	\item[(ii)] $\sum_i\xi(i,k)=1/t$.
\end{itemize}

The constraint (i) formalizes the gradual dose-escalation requirement. Furthermore, in (ii), each cohort is assumed to be of the same size, say $m$.
In the case of extended designs, the only condition for cohort $k=n+1$ is (ii), i.e., the condition (ii) is assumed for all $k \in \{1,\ldots,t\}$. In other words, cohort $n+1$ of extended designs is an additional cohort that is subject to no constraints other than the fixed cohort size.

From now on, we have $\Xi=\{\xi \text{ is feasible} \mid \xi \text{ satisfies (i), (ii)}\}$.
\cite{Bailey09} considered an additional constraint (iii) $\xi(k,k) \geq 1/N$, i.e., in each cohort, the highest allowed dose must be given to at least one subject. Since (iii) will naturally be satisfied by optimal designs, we disregard this constraint.
\bigskip

From condition (ii) follows that $s=1_t/t$; thus, the matrix $M_{22}$ is of rank $t$ and the Schur complement $M_\tau = M_{11} - M_{12}M_{22}^-M_{12}^\T $ of $M$ is
$$M_\tau = \diag(r)-tXX^\T, $$
which may be viewed as the information matrix of $\xi$, when all the treatments are of equal interest. For a general matrix $Q$ of full column rank, the system $Q^\T \tau$ is estimable under $\xi$ if and only if $\mathcal{C}(Q) \subseteq \mathcal{C}(M_\tau)$. The information matrix of a feasible design for estimating $Q^\T  \tau$ is given by $N_A(\xi) = (Q^\T M_\tau^-(\xi)Q)^{-1}$, e.g., see \cite{RosaHarman16}. The inverse of the information matrix $N_A(\xi)$ is proportional to the variance matrix of the least-squares estimator of $Q^T\tau$ under $\xi$; more precisely, $\Var_\xi(\widehat{Q^T\tau}) = N_A^{-1}(\xi) \sigma^2/N$, where $N$ is the total number of subjects.

It is easy to check that the Schur complement for a given design $\xi$, and thus the information matrix of $\xi$, would be the same if the constant term $\mu$ was not present in model (\ref{eModelDE}). It follows that our results hold even in a model without the constant term.

Recall that the experimental aim considered in this paper is the comparison of the drug doses with the placebo. For a feasible design $\xi$, the information matrix for comparing the doses with the placebo (i.e., comparing test treatments with a control) $N_A(\xi)$ is obtained by deleting the first row and column of $M_\tau$ (e.g., see \cite{BechhoferTamhane} or \cite{GiovagnoliWynn}).

Let us denote
\begin{equation}\label{eMtau}
M_\tau = \begin{bmatrix}
\alpha & b^\T  \\ b & C
\end{bmatrix}, \quad
X = \begin{bmatrix}
z^\T  \\ Z
\end{bmatrix}
\end{equation}
where $C$ is an $n \times n$ matrix and $Z$ is an $n \times t$ matrix. Then, 
\begin{equation}\label{eInfMat}
N_A(\xi) = C=\mathrm{diag}(r_1,\ldots,r_n) - tZZ^\T,
\end{equation}
where $Z=\big(\xi(i,k)\big)_{i=1,\ldots,n;k=1,\ldots,t}$.

\section{$E$-optimality}

\subsection{Standard Designs}

A design is $E$-optimal if it maximizes the minimum eigenvalue $\lambda_{\min}(N_A(\xi))$, which is equivalent to minimizing $\lambda_{\max}(N_A^{-1}(\xi))$ (e.g., \cite{puk}; \cite{AtkinsonEA07}).
In order to provide $E$-optimal designs, we characterize the $E$-optimal relative replication numbers (treatment proportions). From Theorem 1 and Theorem 6 by \cite{RosaHarman16}, we obtain the following lemma.

\begin{lemma}\label{lEopt}
	Let $\xi$ be an $E$-optimal design for comparing the drug doses with the placebo in model (\ref{eModelDE}) without design constraints. Then, $\xi$ satisfies $r_0(\xi) = 1/2$ and $r_i(\xi) = 1/(2n)$ for $i=1,\ldots,n$.
\end{lemma}

A (standard) design that satisfies 
$$\xi(0,k)=\xi(k,k)=\frac{1}{2n} \text{ for }k\in\{1,\ldots,n\},$$
and is zero otherwise is a \emph{Senn design} (see \cite{Bailey09}, \cite{Senn97}). In such a design, half of the subjects in each cohort are assigned to the placebo and the other half are assigned to the highest allowed dose for that cohort. Note that Senn designs satisfy Lemma \ref{lEopt}, i.e., they achieve the $E$-optimal treatment proportions.
We show that Senn designs are the unique $E$-optimal standard dose-escalation designs.

\begin{center}
	\begin{table}[h]
		\centering
		\begin{tabular}{  c | c  c  c  c  }
			$i\backslash k$ & 1 & 2 & 3 & 4  \\ \hline
			0 & 1/8 & 1/8 & 1/8 & 1/8  \\
			1 & 1/8 & 0   & 0   & 0  \\
			2 & 0   & 1/8 & 0   & 0  \\
			3 & 0   & 0   & 1/8 & 0 \\
			4 & 0   & 0   & 0   & 1/8 \\ 
		\end{tabular}
		\caption{Senn design for $n=4$ doses. $i$ - treatment number, $k$ - cohort number}
	\end{table}
\end{center}

\begin{theorem}\label{tEoptStd}
	A standard design $\xi$ is $E$-optimal for comparing drug doses with the placebo if and only if $\xi$ is a Senn design.
\end{theorem}

The proof of Theorem \ref{tEoptStd} and all other non-trivial proofs are given in the appendix.

\subsection{Extended Designs}

Since the extra cohort in extended designs provides additional freedom for the experimenter, the class of $E$-optimal extended designs is richer, compared to the single $E$-optimal standard design (for given $n$). The $E$-optimal extended designs can be completely characterized by the following conditions: the designs assign the placebo to half of the subjects in each cohort, and allocate each dose to the same total number of subjects, $N/(2n)$. Note that the $E$-optimality of standard designs can be described by the same conditions; however, in the standard design case, the Senn designs are the only ones that satisfy these conditions. Recall that $t=n+1$ in the extended design case.

\begin{theorem}\label{tEoptExt}
	An extended design $\xi$ is $E$-optimal for comparing the drug doses with the placebo if and only if it satisfies $\xi(0,k)=1/(2t)$ for $k=1,\ldots,n+1$ and $r_1(\xi) = \ldots = r_n(\xi) =1/(2n)$.
\end{theorem}

The standard Senn designs can be extended to obtain $E$-optimal extended designs. We say that a \emph{uniformly extended Senn design} is a design that satisfies
$$\xi(0,k) = \frac{1}{2t}\text{ for } k=1,\ldots,n+1, \quad 
\xi(i,i) = \frac{1}{2t} \text{ and } \xi(i,n+1)=\frac{1}{2nt} \text{ for } i=1,\ldots,n,$$
and is zero otherwise, see Table \ref{tblUnifSenn}. Such a design is given by the standard Senn design in the first $n$ cohorts. In the last cohort, half of the subjects are assigned to the control and the other half are uniformly distributed among all other treatments. The uniformly extended Senn designs obviously satisfy the conditions of Theorem \ref{tEoptExt} and thus are $E$-optimal.

\begin{center}
	\begin{table}[h]
		\centering
		\begin{tabular}{  c | c c c c c  }
			$i\backslash k$ & 1 & 2 & 3 & 4 & 5  \\ \hline
			0 & 1/10 & 1/10 & 1/10 & 1/10 & 1/10 \\
			1 & 1/10 & 0    & 0    & 0    & 1/40  \\
			2 & 0   & 1/10  & 0    & 0    & 1/40  \\
			3 & 0   & 0     & 1/10 & 0    & 1/40  \\
			4 & 0   & 0     & 0    & 1/10 & 1/40 
		\end{tabular}
		\caption{Uniformly extended Senn design for $n=4$ doses. $i$ - treatment number, $k$ - cohort number}
		\label{tblUnifSenn}
	\end{table}
\end{center}

However, the class of $E$-optimal extended dose-escalation designs does not contain only the uniformly extended Senn designs.
The complete characterization of $E$-optimal extended designs in Theorem \ref{tEoptExt} by linear constraints allows one to choose the best $E$-optimal design with respect to some secondary criterion, such as other optimality criteria (e.g., the well known criteria of $D$- and $A$-optimality; for their definitions see \cite{puk}). 

Because the provided characterization of $E$-optimal extended designs is linear, we can employ an efficient convex optimization method from \cite{HarmanSagnol} to construct $A$-optimal or $D$-optimal designs within the class of all $E$-optimal designs:

\begin{example}
	Let $n=4$. Then, the $A$-optimal and $D$-optimal designs in the class of all $E$-optimal extended designs are given in Tables \ref{tblAinEopt} and \ref{tblDinEopt}, respectively.
	
	\begin{center}
		\begin{table}[h]
			\centering
			\begin{tabular}{  c | c  c  c c c }
				$i\backslash k$ & 1 & 2 & 3 & 4 & 5  \\ \hline
				0 & 0.1000 &  0.1000 & 0.1000 & 0.1000 & 0.1000 \\
				1 & 0.1000  & 0.0219 & 0.0031 & 0.0000 & 0.0000 \\
				2 & 0 & 0.0781 & 0.0287 & 0.0091 & 0.0091 \\
				3 & 0 & 0 & 0.0682 & 0.0284 & 0.0284 \\
				4 & 0 & 0 & 0 & 0.0625 & 0.0625
			\end{tabular}
			\caption{$A$-optimal design in the class of all $E$-optimal extended designs. $i$ - treatment number, $k$ - cohort number}
			\label{tblAinEopt}
		\end{table}
	\end{center}
	\begin{center}
		\begin{table}[h]
			\centering
			\begin{tabular}{  c | c  c  c c c }
				$i\backslash k$ & 1 & 2 & 3 & 4 & 5 \\ \hline
				0 & 0.1000  &  0.1000  &  0.1000  &  0.1000 &   0.1000 \\
				1 & 0.1000  &  0.0248  &  0.0002  &  0.0000 &   0.0000 \\
				2 & 0  &  0.0752   & 0.0339  &  0.0079  &  0.0079 \\
				3 & 0  &       0   & 0.0659  &  0.0296  &  0.0296 \\
				4 & 0  &       0   &      0  &  0.0625  &  0.0625
			\end{tabular}
			\caption{$D$-optimal design in the class of all $E$-optimal extended designs. $i$ - treatment number, $k$ - cohort number}
			\label{tblDinEopt}
		\end{table}
	\end{center}
\end{example}

\subsection{Interpretation of $E$-optimality for Dose-Escalation}\label{ssEoptInterpretation}

\cite{MajumdarNotz}, who pioneered the use of the optimality criteria in design of experiments for comparing test treatments with a control, noted that $E$-optimality ``does not seem to posses a very natural statistical meaning [for the comparison with a control]''. Since then, the criterion received little attention in design literature on the comparison with a control, with a few exceptions (see \cite{Notz}, \cite{MorganWang}).

Nevertheless, the $E$-optimality does have a statistical interpretation for estimating a system of contrasts $Q^T\tau$ in terms of the confidence ellipsoid for $Q^T\tau$, which represents the amount of uncertainty about the studied system of contrasts.
Under the assumption of normal errors, the largest eigenvalue of $N_A^{-1}(\xi)$ is proportional to the square of the length of the largest semi-axis of the confidence ellipsoid for $Q^T\tau$. Therefore, any $E$-optimal design for comparing drug doses with the placebo minimizes the length of the largest semi-axis of the confidence ellipsoid for estimating $\tau_1-\tau_0, \ldots, \tau_n-\tau_0$. This suggests that the $E$-optimality criterion may be relevant, as it minimizes the diameter of the confidence ellipsoid.

The $E$-optimality can also be expressed in the terms of the variances of the drug dose-placebo contrasts in a straightforward manner, as by \cite{MorganWang}, even without the normality assumption. We can write
$$\lambda_{\max}(N_A^{-1}(\xi)) = \max_{\lVert x \rVert = 1} x^T N_A^{-1}(\xi) x.$$
Because $N_A^{-1}(\xi)$ is proportional to the variance matrix of the least-squares estimator of $Q^T\tau$ (where $Q^T=(-1_n,I_n)$), we have:

\begin{proposition}\label{pEopt}
	Any standard or extended $E$-optimal design for comparing the drug doses with the placebo minimizes the maximum variance among all least-squares estimators of the drug dose-placebo contrasts of the form  
	\begin{equation}\label{eCwCcontrasts}
	x^T Q^T \tau = \sum_{i=1}^n x_i \tau_i - (\sum_{i=1}^n x_i) \tau_0, \text{ where }\lVert x \rVert = 1.
	\end{equation}
\end{proposition}

\cite{Notz} obtained a class of $E$-optimal exact block designs in a model without design constraints and provided two statistical interpretations for these designs, using test treatment-control contrasts. One of these interpretations (Theorem 3.3) was that the obtained designs minimized the maximum variance for the contrasts of the form \eqref{eCwCcontrasts}. \cite{MorganWang} analysed the relationship as in Proposition \ref{pEopt} further, using weighted optimality criteria (Theorem 2.1 therein).

Employing the relationship between $E$-optimality and $c$-optimality (see \cite{PukStudden}, \cite{puk}), it is possible to give one more interpretation of the $E$-optimality in our model using the drug dose-placebo contrasts.

\begin{theorem}\label{tEoptcopt}
	Let $\xi$ be an $E$-optimal standard or extended design in model (\ref{eModelDE}) for comparing the drug doses with the placebo. Then, $\xi$ is $c$-optimal, where $c=(-1,1_n^\T /n,0_{t+1}^\T )^\T $. That is, $\xi$ minimizes the variance of the least-squares estimator of $\sum_i (\tau_i - \tau_0)/n$.
\end{theorem}

Theorem \ref{tEoptcopt} shows that if a standard or extended design is $E$-optimal for comparing the doses with the placebo, it minimizes the variance of the average dose effect compared with the placebo, or the average of the dose effects relative to the placebo. That is, the $E$-optimal designs (among other desirable properties) also minimize
\begin{equation}\label{ecopt}
\Var\Big(\widehat{\frac{1}{n}\sum_{i=1}^n \tau_i-\tau_0}\Big), 
\end{equation}
where $\hat{\alpha}$ is the least-squares estimator of $\alpha$. The other interpretation of the obtained class of $E$-optimal exact designs in the standard block model by \cite{Notz} was Theorem 3.2, in which the author showed that the obtained designs minimized \eqref{ecopt}.


By contrast, the $A$-optimality criterion, which is often used in test treatment-control experiments, minimizes the average variance of the least-squares estimators of the contrasts of interest. That is, it minimizes the average variance for the dose effects compared with the placebo effect, which can be expressed as minimizing
$$
\frac{1}{n}\sum_{i=1}^n\Var\Big(\widehat{\tau_i-\tau_0}\Big);
$$
compare with (\ref{ecopt}).

In fact, in the proof of Theorem \ref{tEoptcopt} we show that the $E$-optimal designs for comparing drug doses with the placebo minimize $h^T N_Q^{-1}(\xi) h$, where $h=1_n$; i.e., the $E$-optimal designs minimize the sum of all elements of the inverse of the information matrix. It follows that the $E$-optimal designs for comparing drug doses with the placebo minimize the sum of all elements of the variance matrix of $\widehat{Q^T\tau}$. In other words, such designs minimize the sum (or the average) of all variances $\Var(\widehat{\tau_i-\tau_0})$, $i=1,\ldots,n$, and covariances $\mathrm{Cov}(\widehat{\tau_i-\tau_0},\widehat{\tau_j-\tau_0})$, $i \neq j$, for the drug dose-placebo comparisons. By contrast, the $A$-optimal designs minimize the trace of $N_A^{-1}(\xi)$, i.e., the sum (or the average) of all variances $\Var(\widehat{\tau_i-\tau_0})$.

\section{$LV$-optimality}

Senn suggested in his commentary \cite{Senn09} that what he calls the \emph{latest variance} ($LV$) criterion may be more appropriate to consider, rather than the traditional optimality criteria. Such a criterion seeks to minimize the variance of the estimator for comparing the latest dose with the placebo after each cohort. The reasoning is that after each cohort, the experimenter decides whether or not to continue with the experiment, and this decision is based on the results obtained so far (the effect of the latest dose relative to the placebo). Therefore, crucial in the experiment is the variance of the comparison of the latest dose with the placebo, rather than all variances based on trials which may not even be performed (if the experiment stops before all doses were tried). 

Since the Senn designs place all the weight in a given cohort on the comparison of the latest dose with the placebo, intuitively, it is to be expected for them to be $LV$-optimal. We formalize this observation as well as the $LV$-optimality criterion itself.
\bigskip

We denote by $\Var_k(\widehat{\tau_k - \tau_0})$ the variance of the least-squares estimator of $\tau_k - \tau_0$ based on the trials in the first $k$ cohorts, $k=1,\ldots,n$. Then, an $LV$-\emph{optimal} design $\xi$ for comparing the drug doses with the placebo minimizes each of $\Var_1(\widehat{\tau_1 - \tau_0})$, \ldots, $\Var_n(\widehat{\tau_n - \tau_0})$ among all feasible dose-escalation designs. By a feasible design, we mean here a design under which $\tau_1-\tau_0$ is estimable after one cohort, $\tau_2-\tau_0$ is estimable after the first two cohorts etc. We further remark that the $LV$-optimality criterion is not defined for a general design problem; it is strongly tied to the particular properties of the dose-escalation studies.

The criterion of $LV$-optimality seems to call for a simultaneous multicriterial optimization (minimizing $n$ variances) which need not have a solution in general; i.e., some trade-off between the various variances of interest might be needed. However, we show that in our model the Senn designs do have this strong property, i.e., no other dose-escalation design has any of the variances in question lower.

%

\begin{theorem}\label{tLVopt}
	Let $\xi$ be a Senn design. Then, $\xi$ is an $LV$-optimal standard design.
\end{theorem}

In the case of extended designs, a design is $LV$\emph{-optimal} if it minimizes $\Var_1(\widehat{\tau_1 - \tau_0})$, \ldots, $\Var_n(\widehat{\tau_n - \tau_0}), \Var_{n+1}(\widehat{\tau_n - \tau_0})$, where $\Var_{n+1}(\widehat{\tau_n - \tau_0})$ is the variance of $\widehat{\tau_n - \tau_0}$ once the entire study has been carried out. Unsurprisingly, the $LV$-optimal design is the extended Senn design $\xi$ constructed as follows: $\xi$ is given by the Senn design in the first $n$ cohorts and its $(n+1)$-st cohort is a replication of the $n$-th cohort (in the sense of the treatment proportions). Formally, $$\xi(0,k)=\xi(k,k)=\frac{1}{2(n+1)} \text{ for } k=1,\ldots,n, \quad \xi(0,n+1) = \xi(n,n+1) = \frac{1}{2(n+1)},$$
and $\xi(i,k) = 0$ otherwise, see Table \ref{tblHighestSenn}. We call such a design a \emph{highest-dose extended Senn design}.

\begin{center}
	\begin{table}[h]
		\centering
		\begin{tabular}{  c | c  c  c c c  }
			$i\backslash k$ & 1 & 2 & 3 & 4 & 5  \\ \hline
			0 & 1/10 & 1/10 & 1/10 & 1/10 & 1/10 \\
			1 & 1/10 & 0   & 0   & 0 & 0   \\
			2 & 0   & 1/10 & 0   & 0 & 0  \\
			3 & 0   & 0   & 1/10 & 0 & 0 \\
			4 & 0   & 0   & 0 & 1/10 & 1/10 
		\end{tabular}
		\caption{Highest-dose extended Senn design for $n=4$ doses. $i$ - treatment number, $k$ - cohort number}
		\label{tblHighestSenn}
	\end{table}
\end{center}

Note that the extended designs seem to be less meaningful when considering $LV$-optimality, as in such a case, the additional cohort is only used  to improve the estimate of the effect of the highest dose compared with the placebo.

\begin{theorem}\label{tLVoptExt}
	Let $\xi$ be a highest-dose extended Senn design. Then, $\xi$ is an $LV$-optimal extended design.
\end{theorem}

\section{$MV$-optimality}

A widely used optimality criterion in designing experiments for comparing the test treatments with a control is $MV$-optimality. A design is $MV$-optimal if it minimizes the maximum variance of the least-squares estimators of the contrasts of interest. In the case of comparing drug doses with the placebo, a design is $MV$-optimal if it minimizes
$$\Psi(\xi) = \max_{1\leq i \leq n} \Var ( \widehat{\tau_i-\tau_0} ),$$
among all $\xi \in \Xi$.

Let us consider the standard designs. Then, as in the $LV$-optimal case, the Senn designs are optimal.

\begin{theorem}\label{tMVopt}
	Let $\xi$ be a Senn design. Then, $\xi$ is an $MV$-optimal standard design for comparing drug doses with the placebo.
\end{theorem}

From numerical studies it seems that unlike the $MV$-optimal standard designs, the $MV$-optimal extended designs do not follow a regular structure that we are able to describe analytically at this time.


\section{Discussion}
The optimality results for designs of experiments are highly dependent on the chosen system of treatment contrasts of interest. The designs obtained in this paper, which are optimal for the comparison of the drug doses with the placebo, would perform worse than many of the designs proposed by \cite{Bailey09} if the aim of the experiment was to estimate the system of all pairwise comparisons.  Ultimately, the choice of the treatment contrasts of interest lies upon the experimenters. However, following \cite{Senn09}, we noted that the comparison of the drug doses with the placebo may be a relevant system of contrasts for the dose-escalation studies. Therefore, we intended to provide well-performing designs for the case when the dose effects relative to the placebo are the main focus. In such a case, it seems that the Senn designs are a reasonable choice.

Another desirable property of the Senn designs is their extremely simple and regular structure, which ensures a simple construction of exact designs from the approximate Senn designs, and, which may be even more important, the characterization of these designs for practical use. Moreover, as noted by \cite{Senn09}, the Senn designs are not 'anticipatory'. That is, due to their self-similar structure, their optimality properties hold even when the study is stopped and thus the entire design is not performed - e.g., unlike the design given in Table \ref{tblAinEopt} which is $E$-optimal only if the entire experiment is carried out. Such 'non-anticipatory' property is crucial in the $LV$-optimality of the Senn designs.

If the extra cohort is considered and the main focus is on the comparisons of the doses with the placebo, one may choose one of the two extensions of the Senn designs that are provided in this paper. These extensions have forms that are rather simple to describe, and they inherit the optimality properties of the Senn designs if the study is stopped before the entire design is carried out. Alternatively, we provided an entire class of $E$-optimal extended designs described by simple (linear) conditions, which allows one to calculate a design that is optimal in this class with respect to a secondary criterion, if some secondary criterion is deemed relevant.

\section*{Acknowledgements}
This research was supported by the VEGA 1/0521/16 grant of the Slovak Scientific Grant Agency. The research of the first author was additionally supported by the UK/214/2016 grant of the Comenius University in Bratislava. 

\section*{Appendix}

\begin{proof}[Proof of Theorem \ref{tEoptStd}]
	
	The Senn design $\xi$ satisfies $r=(1/2,1_n^\T /(2n))^\T $, $X=(1_n, I_n)^\T/(2n) $ and thus $Z=I_n/(2n)$. Note that $\xi$ is connected (see \cite{EcclestonHedayat}) and therefore feasible. Then, equation \eqref{eInfMat} yields that $N_A=I_n/(4n)$, which has the smallest eigenvalue $1/(4n)$. The proof of Theorem 6 by \cite{RosaHarman16} states that when there are no design constraints, the optimal value of the $E$-criterion is $1/[4(v-1)]$, where $v$ is the number of treatments. In our model, $v=n+1$ and thus the $E$-optimal value is $1/(4n)$, which the Senn design attains.
	
	Now, consider the converse part.
	For a feasible design $\xi$ to be $E$-optimal, it must have the  smallest eigenvalue equal to that of the Senn design, which coincides with the optimal value of $E$-optimality criterion in the model without the design constraints. Thus, such $\xi$ is $E$-optimal in the model without design constraints, and from Lemma \ref{lEopt} it follows that $\xi$ must satisfy 
	\begin{equation}\label{eEoptW}
	r_0(\xi) = 1/2, \quad r_i(\xi) = 1/(2n), \quad i=1,\ldots,n;
	\end{equation} 
	hence $r=(1/2,1_n^\T /(2n))^\T $.
	
	Observe that
	$N_A(\xi) = C = I_n/(2n) - nZZ^\T $, as defined in (\ref{eMtau}).
	Then, the smallest eigenvalue $\lambda_{\min}$ of $N_A(\xi)$ satisfies $\lambda_{\min} = 1/(2n) - n\mu_{\max}$, where $\mu_{\max}$ is the largest eigenvalue of $ZZ^\T $. Let us denote the columns of matrix $Z=(z_{ij})$ by $z_1, \ldots, z_n$. Then,
	$$
	\mu_{\max} 
	= \max_{ \lVert u \rVert =1} u^\T  ZZ^\T  u 
	= \max_{ \lVert u \rVert =1} \lVert Z^\T  u  \rVert^2 
	= \max_{ \lVert u \rVert =1} \sum_{j=1}^{n} (  z_j^\T  u )^2,
	$$
	and for the particular choice of $u=1_n/\sqrt{n}$, we obtain 
	$$
	\mu_{\max} \geq \frac{1}{n} \sum_{j=1}^{n} (  1_n^\T  z_j )^2 = \frac{1}{n} \sum_{j=1}^{n} \Big(\sum_{i=1}^n z_{ij}\Big)^2 = \frac{1}{n} \sum_{j=1}^{n} q_j^2,
	$$
	where $q_j:=\sum_{i=1}^{n}z_{ij}$, $j=1,\ldots,n$. From (i) it follows that $q_j = n^{-1} - \xi(0,j)$, and (\ref{eEoptW}) yields $\sum_j q_j = 1- 1/2 = 1/2$. Therefore, $\mu_{\max}$ is not smaller than the optimal value of the 
	following convex optimization problem:
	\begin{equation}\label{eOptEigenvalue}
	\min_{q\geq0} \frac{1}{n} \sum_{j=1}^{n} q_j^2 \quad\text{subject to}\quad
	\sum_{j=1}^{n} q_j = \frac{1}{2}.
	\end{equation}
	Because the objective function of \eqref{eOptEigenvalue} is strictly convex on the affine hyperplane formed by the equality constraints and invariant with respect to the permutations of $q_1, \ldots, q_n$, the \emph{single} optimal solution $q^*$ of \eqref{eOptEigenvalue} satisfies $q_1^* = \cdots = q_{n}^* = 1/(2n)$. Therefore, the optimal value of \eqref{eOptEigenvalue} is $1/(4n^2)$. It follows that if $q$ does not satisfy $q_1 = \cdots = q_{n} = 1/(2n)$, then $\mu_{\max} > 1/(4n^2)$. Therefore, the smallest eigenvalue $\lambda_{\min}$ of $N_A(\xi)$ for such $\xi$ satisfies $\lambda_{\min} < 1/(4n)$, and thus $\xi$ is not $E$-optimal. Hence, any $E$-optimal design $\xi$ must satisfy $q_1 = \cdots = q_n = 1/(2n)$, i.e., $\xi(0,1) = \cdots = \xi(0,n) = 1/(2n)$. 
	
	From (\ref{eEoptW}) and (i), $\xi$ needs to satisfy $r_i = 1/(2n)$ for $i>0$ and $s_k=\sum_i \xi(i,k) = 1/n$ for all $k$. Using (ii), because $\xi(n,k)$ can be non-zero only for $k=n$, we obtain $\xi(n,n) = 1/(2n)$. Since $\xi(0,n)=1/(2n)$, we have $\xi(i,n) = 0$ for $0<i<n$. Then, $\xi(n-1,k)$ is non-zero only for $k=n-1$, which yields $\xi(n-1,n-1) = 1/(2n)$, $\xi(i,n-1) = 0$ for $0<i<n-1$, etc. It follows that the single $E$-optimal standard design for given $n$ is the Senn design.
\end{proof}

\begin{proof}[Proof of Theorem \ref{tEoptExt}]
	Let $\xi$ be a design that satisfies the conditions of Theorem \ref{tEoptExt}. First, observe that $\xi$ is connected (see \cite{EcclestonHedayat}) and therefore feasible. The partitions of $M_\tau$ defined in (\ref{eMtau}) satisfy $z=1_{n+1}/(2(n+1))$, $X1_{n+1}=r=\big(1/2,1_n^\T /(2n)\big)^\T $, $Z1_{n+1} = 1_n/(2n)$, $X^\T 1_{n+1}=1_{n+1}/(n+1)$ and $Z^\T 1_n = 1_{n+1}/(2(n+1))$.
	
	Thus,
	$N_A = C = I_n/(2n) - (n+1)ZZ^\T $. Moreover,
	$N_A 1_n = 1_n/(4n),$
	i.e., $1_n$ is an eigenvector of $N_A$ with eigenvalue $\lambda_1=1/(4n)$, which is, analogously to the proof of Theorem \ref{tEoptStd}, equal to the optimal smallest eigenvalue in the model without constraints, see \cite{RosaHarman16}. Therefore, to prove that $\xi$ is $E$-optimal, it suffices to prove that $\lambda_1$ is the smallest eigenvalue of $N_A$. The smallest eigenvalue of $N_A$ satisfies $\lambda_{\min}(N_A) = 1/(2n) - (n+1)\mu_{\max}$, where $\mu_{\max}$ is the largest eigenvalue of $ZZ^\T $.
	
	Let $\hat{Z}=\big(2nZ^\T , 1_{n+1}/(n+1)\big)$. The matrix $\hat{Z}$ satisfies $\hat{Z}1_{n+1} = 1_{n+1}$ and $1_{n+1}^\T  \hat{Z} = 1_{n+1}^\T $, i.e., $\hat{Z}$ is doubly stochastic. Using the Birkhoff-von Neumann theorem, we obtain that
	$$\hat{Z} = \sum_{\pi \in \mathfrak{P}_{n+1}} \alpha_\pi P_\pi, $$
	where $\alpha_\pi \geq 0$ for all $\pi$, $\sum_\pi \alpha_\pi = 1$, $\mathfrak{P}_{n+1}$ denotes the set of all permutations of $n+1$ elements and $P_\pi$ is the permutation matrix given by $\pi$. Let us partition the matrices $P_\pi$ as $P_\pi = \big( \tilde{P}_\pi, v_\pi \big)$, where $\tilde{P}_\pi$ is an $(n+1) \times n$ matrix. Then,
	$$
	2nZ^\T  = \sum_{\pi \in \mathfrak{P}_{n+1}} \alpha_\pi \tilde{P}_\pi, \quad Z = \frac{1}{2n} \sum_{\pi \in \mathfrak{P}_{n+1}} \alpha_\pi \tilde{P}_\pi^\T, \quad \frac{1}{n+1}1_{n+1} = \sum_{\pi \in \mathfrak{P}_{n+1}} \alpha_\pi v_\pi. $$
	Let $\mathfrak{P}_{n+1}^{(i)}$ be the set of all permutations $\pi$ of $n+1$ elements, the corresponding matrices of which have their last column equal to $e_i$, i.e., $\pi(n+1)=i$. Then, because the last column of $\hat{Z}$ is $1_{n+1}/(n+1)$, we have $\sum_{\pi \in \mathfrak{P}_{n+1}^{(i)}}\alpha_\pi = 1/(n+1)$.
	
	The largest eigenvalue $\mu_{\max}$ of $ZZ^\T $ is identical to the largest eigenvalue of $Z^\T  Z$; therefore, it satisfies
	$$\begin{aligned}
	\mu_{\max} 
	&= \max_{ \lVert u \rVert =1} u^\T  Z^\T  Z u 
	= \max_{ \lVert u \rVert =1} \lVert Z u  \rVert^2 
	= \frac{1}{4n^2}  \max_{ \lVert u \rVert =1} \Big\lVert  \sum_{\pi \in \mathfrak{P}_{n+1}} \alpha_\pi \tilde{P}_\pi^\T  u  \Big\rVert^2 \\
	&= \frac{1}{4n^2} \max_{ \lVert u \rVert =1} \Big\lVert \sum_{i=1}^{n+1} \sum_{\pi \in \mathfrak{P}_{n+1}^{(i)}} \alpha_\pi \tilde{P}_\pi^\T  u  \Big\rVert^2
	\leq \frac{n+1}{4n^2} \max_{ \lVert u \rVert =1} \sum_{i=1}^{n+1} \Big\lVert \sum_{\pi \in \mathfrak{P}_{n+1}^{(i)}} \alpha_\pi \tilde{P}_\pi^\T  u  \Big\rVert^2,
	\end{aligned}$$
	where the last inequality follows from the fact that $\lVert \sum_{i=1}^{k} x_i \rVert^2 \leq k \sum_{i=1}^{k} \lVert x_i \rVert^2$ for any vectors $x_1, \ldots, x_k$.
	
	Let $i \in \{1,\ldots,n+1\}$. Then the set of the extremal points of the convex set
	$$
	\mathfrak{K}_i:=\Big\{\sum_{\pi \in \mathfrak{P}_{n+1}^{(i)}} \alpha_\pi \tilde{P}_\pi^\T  u \,\Big\vert\, \alpha_{\pi} \geq 0 \text{ for } \pi \in \mathfrak{P}_{n+1}^{(i)}, \sum_{\pi \in \mathfrak{P}_{n+1}^{(i)}} \alpha_\pi = \frac{1}{n+1}  \Big\}
	$$
	is a subset of $\{ \tilde{P}_{\pi}^T u /(n+1) \,\vert\, \pi \in \mathfrak{P}_{n+1}^{(i)} \}$, and the convex function $\lVert x \rVert^2, x \in \mathfrak{K}_i$, attains its maximum in at least one of the extremal points of $\mathfrak{K}_i$. It follows that
	$$
	\Big\lVert \sum_{\pi \in \mathfrak{P}_{n+1}^{(i)}} \alpha_\pi \tilde{P}_\pi^\T  u  \Big\rVert^2 \leq 
	\max_{ \pi_i \in \mathfrak{P}_{n+1}^{(i)} } \Big\lVert \frac{1}{n+1} \tilde{P}_{\pi_i}^\T  u  \Big\rVert^2 =
	\frac{1}{(n+1)^2} \max_{ \pi_i \in \mathfrak{P}_{n+1}^{(i)} } \lVert \tilde{P}_{\pi_i}^\T  u  \rVert^2;
	$$
	thus,
	$$
	\mu_{\max} \leq
	\frac{1}{4n^2(n+1)} \max_{ \lVert u \rVert =1} \max_{ \pi_1 \in \mathfrak{P}_{n+1}^{(1)},\ldots, \pi_{n+1} \in \mathfrak{P}_{n+1}^{(n+1)} }  \sum_{i=1}^{n+1} \lVert \tilde{P}_{\pi_i}^\T  u  \rVert^2.
	$$
	Consider $\pi_1 \in \mathfrak{P}_{n+1}^{(1)},\ldots, \pi_{n+1} \in \mathfrak{P}_{n+1}^{(n+1)}$ and $u \in \mathbb{R}^{n+1}$, $\lVert u \rVert =1$, that maximize $\sum_{i=1}^{n+1} \lVert \tilde{P}_{\pi_i}^\T  u  \rVert^2$,  and let us denote $v_i=(v_{i,1},\ldots,v_{i,n})^\T :=\tilde{P}^\T _{\pi_i} u, $ $i=1,\ldots,n+1$. Then, $\mu_{\max}$ satisfies
	$$
	\mu_{\max} 
	\leq \frac{1}{4n^2(n+1)} \sum_{i=1}^{n+1} \sum_{j=1}^n v_{i,j}^2 = \frac{1}{4n^2(n+1)} \sum_{i=1}^{n+1} S,
	$$
	where $S$ is the sum of squares of all elements of vectors $v_1,\ldots,v_{n+1}$. For each $i \in \{1,\ldots,n+1\}$, the vector $v_i=\tilde{P}_{\pi_i}^\T  u$ consists of all elements of $u$ except $u_i$; therefore, $u_i^2$ occurs $n$-times in $S$. Hence, $S= nu_1^2 + \cdots + nu_{n+1}^2 = n \lVert u \rVert^2 = n$ and it follows that $\mu_{\max} \leq (4n(n+1))^{-1}$.
	Since $\mu_{\max}\leq (4n(n+1))^{-1}$, the smallest eigenvalue of $N_A$ is at least $ 1/(2n) - (n+1)/(4n(n+1)) = 1/(4n)$ and thus $\lambda_1$ is indeed the smallest eigenvalue of $N_A$.
	
	To prove the converse part, note that similarly to Theorem \ref{tEoptStd}, from Lemma \ref{lEopt} it follows that for $\xi$ to attain the optimal value of $E$-criterion, $1/(4n)$, it must satisfy (\ref{eEoptW}). 
	Let $\xi$ be a design that satisfies (\ref{eEoptW}) but does not satisfy
	$\xi(0,k) = 1/(2(n+1))$, $k=1,\ldots,n+1.$
	Let $q_j:= 1/(n+1) - \xi(0,j)$, $j=1,\ldots,n+1$; hence, $\sum_j q_j = 1-1/2=1/2$. Then, analogously to the proof of Theorem \ref{tEoptStd}, the largest eigenvalue $\mu_{\max}$ of $ZZ^\T $ is not smaller than the optimal value of the following convex optimization problem:
	$$
	\min_{q \geq 0} \frac{1}{n} \sum_{j=1}^{n+1} q_j^2 \quad\text{subject to}\quad
	\sum_{j=1}^{n+1} q_j = \frac{1}{2},
	$$
	which has a unique optimal solution $q^*$ satisfying $q_1^* = \cdots = q_{n+1}^* = (2(n+1))^{-1}$, and the optimal value is $(4n(n+1))^{-1}$. It follows that if $q$ does not satisfy $q_1 = \cdots = q_{n+1} = (2(n+1))^{-1}$, then $\mu_{\max} > (4n(n+1))^{-1}$. Therefore, the smallest eigenvalue $\lambda_{\min}$ of $N_A$ for such $\xi$ satisfies $\lambda_{\min} < (4n)^{-1}$ and thus $\xi$ is not $E$-optimal.
\end{proof}

In the proof of Theorem \ref{tEoptcopt}, Theorems 7.21 (the General Equivalence Theorem for $E$-optimality) and 7.23 by \cite{puk} will be used.

\begin{lemma}[Theorem 7.21 from \cite{puk}]
	A feasible design $\xi$ with its moment matrix $M$ and information matrix $N_A$ is 
	$E$-optimal for estimating $A^\T \beta$ if and only if there exist a generalized inverse $G$ of $M$ and a non-negative definite matrix $E$, $\mathrm{tr}(E)=1$, such that
	\begin{equation}\label{eGETE}
		\mathrm{tr}(M(\tilde{\xi})GAN_AEN_AA^\T  G^\T  ) \leq \lambda_{\min}(N_A) \text{ for all } \tilde{\xi} \in \Xi.
	\end{equation}
\end{lemma}

\begin{lemma}[Theorem 7.23 from \cite{puk}]\label{lEoptcopt}
	Let $\xi$ be a feasible design for $A^\T \beta$ with its information matrix $N_A$ and let $h$, $\lVert h \rVert = 1$, be an eigenvector of $N_A$ corresponding to the smallest eigenvalue of $N_A$. Then, $\xi$ is $E$-optimal for $A^\T \beta$ and $E=hh^\T $ satisfies (\ref{eGETE}) if and only if $\xi$ is $c$-optimal, where $c=Ah$. If the smallest eigenvalue of $N_A$ has multiplicity 1, $\xi$ is $E$-optimal for $A^\T \beta$ if and only if it is $c$-optimal, where $c=Ah$.
\end{lemma}

\begin{proof}[Proof of Theorem \ref{tEoptcopt}]
	Recall that $A^T=(Q^T,0_{n \times (t+1)})$ and $Q^T=(-1_n,I_n)$.
	Let $\tilde{c}:=(-1,1_n^\T /n,0_{t+1}^\T )^\T $.
	For the (standard) Senn design, we have $N_A=I_n/(4n)$. Let us choose
	\begin{equation}\label{eGInv}
	G=\begin{bmatrix}
	M_\tau^- & - M_\tau^-M_{12}M_{22}^- \\ 
	-M_{22}^- M_{12}^\T  M_\tau^- &  M_{22}^- + M_{22}^-M_{12}^\T  M_\tau^-M_{12}M_{22}^-
	\end{bmatrix},
	\end{equation}
	where
	\begin{equation}\label{eGInvM22}
	M_\tau^- = \begin{bmatrix}
	0 & 0_n^\T  \\ 0_n & 4nI_n
	\end{bmatrix} \text{ and }
	M_{22}^- =
	\begin{bmatrix}
	0 & 0_n^\T  \\ 0_n & nI_n
	\end{bmatrix},
	\end{equation}
	and $E=1_n1_n^\T /n = hh^\T $, where $h=1_n/\sqrt{n}$. The matrix $G$ is indeed a generalized inverse of $M$, see Theorem 9.6.1 by \cite{Harville}. Then, it is possible to verify that there is equality in \eqref{eGETE} for any permissible design $\tilde{\xi}$. Therefore, Lemma \ref{lEoptcopt} implies that any $E$-optimal standard design (i.e., Senn design) is also $c$-optimal, where $c= A1_n/\sqrt{n}$, which is equivalent to $A1_n/n=:\tilde{c}$.
	
	Let $\xi$ be an $E$-optimal extended design for comparing the treatments with the placebo, with its moment matrix $M$ and information matrix $N_A$. Then, $h=1_{n}/\sqrt{n}$ is an eigenvector of $N_A$ corresponding to the smallest eigenvalue $\lambda_{\min}=1/(4n)$ (see the proof of Theorem \ref{tEoptExt}). The left-hand side of the normality inequality (\ref{eGETE}) for $E=hh^T$ and a given $\tilde{\xi} \in \Xi$ becomes
	\begin{equation}\label{eGETlhs}
		\mathrm{tr}(\tilde{M}GAN_Ahh^TN_AA^TG^T)=\lambda_{\min}^2h^TA^TG^T\tilde{M}GAh = \frac{1}{16n^3}1_n^TA^TG^T\tilde{M}GA1_n,
	\end{equation}
	where $\tilde{M}:=M(\tilde{\xi})$. Moreover, $A1_n=(-n,1_n^T,0_{n+2}^T)^T$.
	
	Recall the partitioning of $M_\tau$ and $X$ defined in \eqref{eMtau}.
	Let the generalized inverse $G$ of $M$ be given by (\ref{eGInv}), where
	$$
	M_\tau^-=\begin{bmatrix}
	0 & 0_n^T \\ 0_n & C^{-1}
	\end{bmatrix}.
	$$
	Then,
	$$GA1_n = \begin{bmatrix}
	0 \\ C^{-1}1_n \\ 0 \\ -(n+1)Z^TC^{-1}1_n
	\end{bmatrix}.
	$$
	Since $C^{-1}=N_A^{-1}$, we obtain $C^{-1}1_n = 4n1_n$. Then, $GA1_n=(0, 4n1_n^T, 0, -(n+1)4n 1_n^T Z )^T = (0, 4n1_n^T, 0, -2n1_{n+1}^T )^T$, because $Z^T1_n = 1_{n+1}/(2(n+1))$. 
	
	Let us partition $\tilde{M}=M(\tilde{\xi})$ using $\tilde{\alpha}$, $\tilde{b}$, $\tilde{z}$ and $\tilde{Z}$ as in \eqref{eMtau}, and let us denote $\tilde{r}=r(\tilde{\xi})$.
	Then, (\ref{eGETlhs}) becomes 
	$$\frac{1}{16n^3}\Big( (0,4n1_n^T)\tilde{M}_{11}\begin{pmatrix}
	0 \\ 4n1_n
	\end{pmatrix} + (0,-2n1_{n+1}^T)\tilde{M}_{22}\begin{pmatrix}
	0 \\ -2n1_{n+1}
	\end{pmatrix} + 2 (0,4n1_n^T) \tilde{M}_{12} \begin{pmatrix}
	0 \\ -2n1_{n+1}
	\end{pmatrix} \Big), $$
	which is equal to
	$$
	\frac{1}{16n^3}\Big( 16n^2 \sum_{i>0}\tilde{r}_i + 4n^2 -16n^2\sum_{i>0}\tilde{r}_i
	\Big) = \frac{1}{4n} =\lambda_{\min},
	$$
	because $\tilde{Z}1_{n+1}=(\tilde{r}_1, \ldots, \tilde{r}_n)^T$. Therefore, the left-hand side of (\ref{eGETE}) is always equal to the right-hand side. Hence, Lemma \ref{lEoptcopt} yields that $\xi$ is $\tilde{c}$-optimal.
\end{proof}

\begin{proof}[Proof of Theorem \ref{tLVopt}]
	
	Let $\xi$ be a standard design and let $k \in \{1,\ldots,n\}$.
	The variance matrix of the least-squares estimator of the contrasts of interest, $\Var_\xi(\widehat{Q^\T  \tau})$, is proportional to $N_A^{-1}(\xi)=Q^TM_\tau^-Q$.
	The variance calculated from the first $k$ cohorts under a given design $\xi$, $\Var_k(\widehat{\tau_k - \tau_0})$, can be obtained from
	a stage-$k$ 'design' $\xi^{(k)}$. Such a design is given by deleting all the trials in cohorts $k+1,\ldots,n$ from $\xi$, i.e., for any $i$: $\xi^{(k)}(i,j) = \xi(i,j)$ for $j\leq k$ and $\xi^{(k)}(i,j) = 0$ for $j>k$. Note that for $k<n$, $\xi^{(k)}$ is not a proper approximate design, because all of its elements do not sum to $1$, indicating that the trials in the future cohorts are ignored.
	
	The  stage-$k$ moment matrix $M^{(k)}$ given by $\xi^{(k)}$ satisfies
	$M_{11}^{(k)}=\diag(r_0,\ldots,r_k,0_{n-k}^\T )$,
	$$M_{12}^{(k)}=\begin{bmatrix}
	\begin{bmatrix}
	r_0 \\ \vdots \\ r_k
	\end{bmatrix}
	&
	\begin{bmatrix}
	\xi(0,1) & \xi(0,2) & \ldots & \xi(0,k) \\
	\xi(1,1) & \xi(1,2) & \ldots & \xi(1,k) \\
	0 & \xi(2,2) & \ldots & \xi(2,k) \\
	\vdots & \vdots & & \vdots \\
	0 & 0 & \ldots & \xi(k,k)
	\end{bmatrix}
	& 0_{(k+1) \times (t-k)}
	\\ 0_{n-k} & 0_{(n-k) \times k} & 0_{(n-k) \times (t-k)}
	\end{bmatrix}$$
	and
	$$M_{22}^{(k)} = \begin{bmatrix}
	1 & 1_k^\T /n & 0_{t-k}^\T  \\
	1_k/n & I_k/n & 0_{k \times (t-k)} \\
	0_{(t-k)} & 0_{(t-k) \times k} & 0_{(t-k) \times (t-k)}
	\end{bmatrix},$$
	where $t=n$, because we are considering standard designs.
	Then, the latest variance $\Var_k(\widehat{\tau_k - \tau_0})$ is proportional to $d_k(\xi):=(e_{k+1}^\T  - e_1^\T , 0_{t+1}^\T ) \big(M^{(k)}\big)^- (e_{k+1}^\T  - e_1^\T , 0_{t+1}^\T )^\T $, where $e_1$ and $e_{k+1}$ are elementary unit vectors of length $n+1$.
	
	To calculate $d_k(\xi)$, we  disregard cohorts $k+1,\ldots,n$ and we are not allowed to use doses $k+1,\ldots,n$ in the first $k$ cohorts. Such a model coincides with model (\ref{eModelDE}), where $n=k$ and a design $\xi'$ given by $\xi$ for doses $0,\ldots,k$ and cohorts $1,\ldots,k$.
	Then the latest variance under $\xi$ is proportional to the inverse of the value of the $c$-optimality criterion for $\xi'$, which is $\Phi_c(\xi') = \big( c^\T  M^-(\xi') c \big)^{-1}$ for $c^\T =(e_{n+1}^\T -e_1^\T ,0_{n+1}^\T )$. Therefore, without loss of generality, we may assume that $k=n$ and prove that the Senn design has the highest value of $c$-optimality criterion, where $c^\T =(e_{n+1}^\T -e_1^\T ,0_{n+1}^\T )$.
	
	
	
	The General Equivalence Theorem in the case of $c$-optimality becomes (Corollary 5.1 from \cite{Pukelsheim80}): Let $\mathfrak{M}$ be a set of competing moment matrices. The moment matrix $M \in \mathfrak{M}$ is $c$-optimal in $\mathfrak{M}$ if and only if there exists a generalized inverse $G$ of $M$, such that $c^\T  G^\T  BGc \leq c^\T  M^-c$ for all $B \in \mathfrak{M}$.
	
	Now, let $\xi$ be a Senn design, let $c=(e_{n+1}^\T -e_1^\T ,0_{n+1}^\T )^\T $ and let us denote $M:=M(\xi)$. Then, $r=2^{-1}\big(1,1_n^\T /n)^\T $, $X=(2n)^{-1}\big(1_n, I_n\big)^\T $,
	$$
	M_\tau = \frac{1}{4n}
	\begin{bmatrix}
	n & -1_n^\T  \\
	-1_n & I_n
	\end{bmatrix},
	\text{ and let }
	M_\tau^- = \begin{bmatrix}
	0 & 0_n^\T  \\
	0_n & 4nI_n
	\end{bmatrix}.
	$$
	Thus, $c^\T  M^- c = (e_{n+1}-e_1)^\T  M_\tau^- (e_{n+1}-e_1) = 4n$.
	
	Let $G$ be given by (\ref{eGInv}), where $M_{22}^-$ is given by (\ref{eGInvM22}).
	It follows that 
	$$Gc
	=\begin{bmatrix}
	0 & 0_n^\T  & 
	\multirow{2}{*}{$G_{12}$} \\ 
	0_n & 4nI_n  \\
	0 & 0_n^\T  & \multirow{2}{*}{$G_{22}$} \\
	0_n & -2nI_n 
	\end{bmatrix} \begin{bmatrix}
	-1 \\ 0_{n-1} \\ 1 \\ 0_{n+1}
	\end{bmatrix}
	=\begin{bmatrix}0_n \\ 4n \\ 0_n \\ -2n \end{bmatrix}$$
	and therefore, for any feasible design $\xi'$ satisfying conditions (i), (ii), we have 
	
	$$\begin{aligned}
	c^\T  G^\T  M(\xi')Gc 
	&= 
	\begin{bmatrix}
	0_n^\T  & 4n & 0_n^\T  & -2n
	\end{bmatrix}
	\begin{bmatrix}
	\diag(r') & r' & X' \\ {r'}^T & 1 & 1_n^T/n \\ {X'}^T & 1_n/n & I_n/n
	\end{bmatrix}
	\begin{bmatrix} 0_n \\ 4n \\ 0_n \\ -2n \end{bmatrix} \\
	&= (4n)^2 r_n' + (-2n)^2 \frac{1}{n} + 2 \times 4n (-2n) \xi'(n,n) = 4n,
	\end{aligned}$$
	because $\xi'(n,n) = r_n'$.
	Therefore, any design satisfying conditions (i), (ii) satisfies the desired inequality $4n \leq c^\T  M^-c = 4n$. Hence, for any $k$, the Senn design attains the minimum possible latest variance.
\end{proof}

\begin{proof}[Proof of Theorem \ref{tLVoptExt}]
	For the first $n$ latest variances, the argument is the same as in the standard design case, see the proof of Theorem \ref{tLVopt}. Optimality of $\xi$ with respect to the $(n+1)$-st latest variance can be proved by using the General Equivalence Theorem for $c$-optimality, $c=e_{n+1}-e_1$. 
	
	Design $\xi$ satisfies
	$$X=\frac{1}{2(n+1)}\begin{bmatrix}
	1_{n-1}^\T  & 1_2^\T  \\ I_{n-1} & 0_{(n-1)\times 2} \\ 0_{n-1}^\T  & 1_2^\T   
	\end{bmatrix},\quad
	r=\begin{bmatrix}
	\frac{1}{2} \\ \frac{1}{2(n+1)}1_{n-1} \\ \frac{1}{n+1}
	\end{bmatrix}.$$
	Then,
	$$M_\tau 
	=\begin{bmatrix}
	1/4 & -\frac{1}{4(n+1)}1_{n-1}^\T  & -\frac{1}{2(n+1)} \\
	-\frac{1}{4(n+1)}1_{n-1} & \frac{1}{4(n+1)}I_{n-1} & 0_{n-1} \\
	-\frac{1}{2(n+1)} & 0_{n-1}^\T  & \frac{1}{2(n+1)}
	\end{bmatrix}
	$$
	and
	$$M_\tau^-
	=2(n+1)\begin{bmatrix}
	0 & 0_{n-1}^\T  & 0 \\ 0_{n-1} & 2I_{n-1} & 0_{n-1} \\ 0 & 0_{n-1}^\T  & 1
	\end{bmatrix}$$
	is a generalized inverse of $M_\tau$.
	Hence, $c^\T  M^- c =2(n+1)$. Furthermore,
	$$
	M_{22}^- = \begin{bmatrix}
	0 & 0_{n+1}^\T  \\ 0_{n+1} & (n+1)I_{n+1}
	\end{bmatrix}
	$$
	is a generalized inverse of $M_{22}$.
	By choosing the generalized inverse $G$
	of $M$ given by (\ref{eGInv}), after some calculations similar to the proof of Theorem \ref{tLVopt}, the normality inequality becomes $c^\T G^\T M(\xi')Gc = 2(n+1) \leq c^\T  M^-c = 2(n+1)$ for any design $\xi'$ satisfying (i), (ii).
\end{proof}

\begin{proof}[Proof of Theorem \ref{tMVopt}]
	Let $\xi$ be a standard design and let us denote $v_i(\xi):=\Var_\xi(\widehat{\tau_i-\tau_0})$, $i=1,\ldots,n$. Then, $\Psi(\xi) = \max_i v_i(\xi)$.
	Theorem \ref{tLVopt} states that the Senn design $\xi_S$ is $LV$-optimal, i.e., $d_k(\xi) \geq d_k(\xi_S)$ for all $k=1,\ldots,n$. If $k=n$, then the latest variance and the final variance (the variance when the entire design is carried out) coincide. It follows that $v_n(\xi) \geq v_n(\xi_S)$.
	
	In the proof of Theorem \ref{tEoptStd}, we can observe that $Q^\T M_\tau^-(\xi_S)Q = 4nI_n$. It follows that $v_1(\xi_S) = \ldots = v_n(\xi_S) \approx 4n$ and thus $\Psi(\xi_S) = v_n(\xi_S)$. Because $v_n(\xi) \geq v_n(\xi_S)$, we obtain $\Psi(\xi) \geq v_n(\xi) \geq \Psi(\xi_S)$.
\end{proof}

\bibliographystyle{plainnat}
\bibliography{C:/BibTeX/rosa.bib}

\end{document}